\documentclass[reqno, 12pt]{amsart}
\pdfoutput=1
\makeatletter
\let\origsection=\section \def\section{\@ifstar{\origsection*}{\mysection}} 
\def\mysection{\@startsection{section}{1}\z@{.7\linespacing\@plus\linespacing}{.5\linespacing}{\normalfont\scshape\centering\S}}
\makeatother        

\usepackage{amsmath,amssymb,amsthm}
\usepackage{mathrsfs}
\usepackage{mathabx}\changenotsign
\usepackage{dsfont}

\usepackage{xcolor}
\usepackage[backref]{hyperref}
\hypersetup{
    colorlinks,
    linkcolor={red!60!black},
    citecolor={green!60!black},
    urlcolor={blue!60!black}
}

\usepackage{graphicx}

\usepackage{verbatim}

\usepackage[open,openlevel=2,atend]{bookmark}

\usepackage[abbrev,msc-links,backrefs]{amsrefs} 
\usepackage{doi}

\renewcommand{\PrintDOI}[1]{\doi{#1}}

\usepackage[T1]{fontenc}
\usepackage{lmodern}
\usepackage[babel]{microtype}
\usepackage[english]{babel}

\usepackage{geometry}
\numberwithin{equation}{section}
\numberwithin{figure}{section}

\usepackage{enumitem}

\let\polishlcross=\l
\def\l{\ifmmode\ell\else\polishlcross\fi}

\def\paragraph#1{%
  \noindent\textbf{#1.}\enspace}

\let\emptyset=\varnothing
\let\setminus=\smallsetminus

\makeatletter
\def\moverlay{\mathpalette\mov@rlay}
\def\mov@rlay#1#2{\leavevmode\vtop{   \baselineskip\z@skip \lineskiplimit-\maxdimen
   \ialign{\hfil$\m@th#1##$\hfil\cr#2\crcr}}}
\newcommand{\charfusion}[3][\mathord]{
    #1{\ifx#1\mathop\vphantom{#2}\fi
        \mathpalette\mov@rlay{#2\cr#3}
      }
    \ifx#1\mathop\expandafter\displaylimits\fi}
\makeatother

\DeclareFontFamily{U}  {MnSymbolC}{}
\DeclareSymbolFont{MnSyC}         {U}  {MnSymbolC}{m}{n}
\DeclareFontShape{U}{MnSymbolC}{m}{n}{
    <-6>  MnSymbolC5
   <6-7>  MnSymbolC6
   <7-8>  MnSymbolC7
   <8-9>  MnSymbolC8
   <9-10> MnSymbolC9
  <10-12> MnSymbolC10
  <12->   MnSymbolC12}{}
\DeclareMathSymbol{\powerset}{\mathord}{MnSyC}{180}

\let\epsilon=\varepsilon

\let\rho=\varrho
\let\theta=\vartheta

\theoremstyle{plain}
\newtheorem{thm}{Theorem}[section]
\newtheorem{theorem}[thm]{Theorem}
\newtheorem{lemma}[thm]{Lemma}
\newtheorem{corollary}[thm]{Corollary}
\newtheorem{proposition}[thm]{Proposition}

\newtheorem{thm-intro}{Theorem}[]
\newtheorem*{claim*}{Claim}

\theoremstyle{definition}

\newtheorem{remark}[thm]{Remark}

\newtheorem{example}[thm]{Example}

\usepackage{accents}

\let\phi=\varphi

\newcommand{\no}[1]{}
\newcommand{\abs}[1]{\ensuremath{{\lvert {#1} \rvert}}}

\begin{document}

\author{J. Pascal Gollin and Karl Heuer}
\address{J. Pascal Gollin, Discrete Mathematics Group, Institute for Basic Science (IBS), 55, Expo-ro, Yuseong-gu, 34126 Daejeon, Republic of Korea}
\email{\tt pascalgollin@ibs.re.kr}
\address{Karl Heuer, Institute of Software Engineering and Theoretical Computer Science, Technische Universit\"{a}t Berlin, Ernst-Reuter-Platz 7, 10587 Berlin, Germany}
\email{\tt karl.heuer@tu-berlin.de}

\title[]{An analogue of Edmonds' Branching Theorem for infinite digraphs}

\begin{abstract}
We extend Edmonds' Branching Theorem to locally finite infinite digraphs.
As examples of Oxley or Aharoni and Thomassen show, this cannot be done using ordinary arborescences, whose underlying graphs are trees.
Instead we introduce the notion of pseudo-arborescences and prove a corresponding packing result.
Finally, we verify some tree-like properties for these objects, but give also an example that their underlying graphs do in general not correspond to topological trees in the Freudenthal compactification of the underlying multigraph of the digraph.
\end{abstract}

\maketitle

\setcounter{footnote}{1}

\section{Introduction}
\label{sec:intro}

Studying how to force spanning structures in finite graphs is a basic task. 
The most fundamental spanning structure is a spanning tree, whose existence is already characterised by the connectedness of the graph. 
Moving on and characterising the existence of a given number of edge-disjoint spanning trees via an immediately necessary condition, Nash-Williams~\cite{nash-williams} and Tutte~\cite{tutte} independently proved the following famous theorem.

\begin{theorem}
    \label{thm:tree-packing}
    \cites{nash-williams, tutte}, 
    \cite{diestel_buch}*{Theorem~2.4.1}
    A finite multigraph~$G$ has~${k \in \mathbb{N}}$ edge-disjoint spanning trees 
    if and only if 
    for every partition~$\mathcal{P}$ of~${V(G)}$ there are at least ${k(|\mathcal{P}| - 1)}$ edges in~$G$ whose endvertices lie in different partition classes. 
\end{theorem}

Later, Edmonds~\cite{edmonds} generalised Theorem~\ref{thm:tree-packing} to finite digraphs, also involving a condition which is immediately seen to be necessary for the existence of the spanning structures. 
In his theorem, Edmonds considers as spanning structures \emph{out-arborescences} rooted in a vertex~$r$, i.e.~spanning trees whose edges are directed away from the root~$r$. 
His theorem immediately implies a corresponding result for \emph{in-arborescences} rooted in~$r$, i.e.~spanning trees directed towards~$r$, via reversing every edge in the digraph. 
For this reason we shall focus in this paper only on out-arborescences and denote them just by \emph{arborescences}. 

\begin{theorem}
    \label{thm:fin-edmonds}
    \cite{edmonds}, 
    \cite{bang-jensen}*{Theorem~9.5.1}
    A finite digraph~$G$ with a vertex~${r \in V(G)}$ has~${k \in \mathbb{N}}$ edge-disjoint spanning arborescences rooted in~$r$ 
    if and only if 
    there are at least~$k$ edges from~$X$ to~$Y$ for every bipartition~${(X, Y)}$ of~${V(G)}$ with~${r \in X}$. 
\end{theorem}

One of the main results of this paper is to extend Theorem~\ref{thm:fin-edmonds} to a certain class of infinite digraphs. 
There has already been work in this area. 
In order to mention two important results about this let us call a one-way infinite path all of whose edges are directed away from the unique vertex incident with only one edge a \emph{forwards directed ray}. 
Similarly, we call the digraph obtained by reversing all edges of a forwards directed ray a \emph{backwards directed ray}. 
Thomassen~\cite{Thomassen:edmonds} extended Theorem~\ref{thm:fin-edmonds} to infinite digraphs that do not contain a backwards directed ray, while Jo\'o~\cite{Joo:Edmonds} obtained an extension for infinite digraphs without forwards directed rays using different methods. 
In contrast to these two results we shall demand a local property for our digraphs by considering \emph{locally finite digraphs}, i.e.~digraphs in which every vertex has finite in- and out-degree. 
Similarly, undirected multigraphs are called \emph{locally finite} if every vertex has finite degree. 

When trying to extend Theorem~\ref{thm:fin-edmonds} to infinite digraphs it is important to know that a complete extension is not possible. 
The reason for this is that Oxley~\cite{oxley}*{Example~2} constructed a locally finite graph without two edge-disjoint spanning trees but fulfilling the
condition in Theorem~\ref{thm:tree-packing}. 
Following up, Aharoni and Thomassen~\cite{aharoni-thomassen}*{Theorem} gave a construction for further counterexamples to Theorem~\ref{thm:fin-edmonds}, which are all locally finite and can even be made $2k$-edge-connected for arbitrary~${k \in \mathbb{N}}$. 
Hence, using ordinary spanning trees for an extension of Theorem~\ref{thm:tree-packing} to locally finite graphs does not work. 
This immediately implies that extending Theorem~\ref{thm:fin-edmonds} to locally finite digraphs fails as well if ordinary arborescences are used. 
While Thomassen and Jo\'o could overcome this problem by forbidding certain one-way infinite paths, for us it is necessary to additionally change the notion of arborescence since the counterexamples to direct extensions of Theorem~\ref{thm:tree-packing} and Theorem~\ref{thm:fin-edmonds} to infinite (di)graphs are locally finite.

For undirected locally finite (connected) multigraphs~$G$ the problem of how to extend Theorem~\ref{thm:tree-packing} has successfully been overcome. 
The key was to not just consider~$G$ but the Freudenthal compactification ${\abs{G}}$~\cites{diestel_buch, Diestel:topGTsurvey} of the $1$-complex of~$G$. 
Instead of ordinary spanning trees, now packings of topological spanning trees of~$G$ are considered. 
We call a connected subspace of~$\abs{G}$ which is the closure of a set of edges of~$G$, contains all vertices of~$G$ but contains no homeomorphic image of the unit circle~${S^1 \subseteq \mathbb{R}^2}$, a \emph{topological spanning tree} of~$G$. 
There is an equivalent but more combinatorial, and in particular finitary, way of defining topological spanning trees of~$G$. 
They are precisely the closures in~$\abs{G}$ of the minimal edge sets that meet every finite cut of~$G$~\cite{diestel_buch}. 
As already observed by Tutte, this finitary condition can be used to obtain the following packing theorem for disjoint edge sets each meeting every finite cut, via the compactness principle. 

\begin{theorem}
    \label{thm:inf-tutte-packing}
    \cite{tutte}
    A locally finite multigraph~$G$ has~${k \in \mathbb{N}}$ disjoint edge sets each meeting every finite cut of~$G$ 
    if and only if 
    for every finite partition~$\mathcal{P}$ of~${V(G)}$ there are at least~${k(|\mathcal{P}| - 1)}$ edges in~$G$ whose endvertices lie in different partition classes. 
\end{theorem}

By the equivalence noted above, Theorem~\ref{thm:inf-tutte-packing} implies a packing result for topological spanning trees: 

\begin{theorem}
    \label{thm:tst-packing}
    \cite{diestel_buch}*{Theorem~8.5.7}
    A locally finite multigraph~$G$ has~${k \in \mathbb{N}}$ edge-disjoint topological spanning trees 
    if and only if 
    for every finite partition~$\mathcal{P}$ of~${V(G)}$ there are at least~${k(|\mathcal{P}| - 1)}$ edges in~$G$ whose endvertices lie in different partition classes. 
\end{theorem}

In the spirit of Tutte's approach, we prove the following packing theorem generalising Theorem~\ref{thm:fin-edmonds} to locally finite digraphs for what we call spanning pseudo-arborescence rooted in some vertex~$r$. 
For a locally finite weakly connected digraph~$G$ and~${r \in V(G)}$ we define a \emph{spanning pseudo-arborescence rooted in~$r$} as a minimal edge set ${F \subseteq E(G)}$ such that~$F$ contains, 
for every bipartition~${(X, Y)}$ of~${V(G)}$ with~${r \in X}$ and finitely many edges between~$X$ and~$Y$ in either direction, 
an edge directed from~$X$ to~$Y$.

\begin{theorem}
    \label{thm:weak-pseudo-edmonds}
    A locally finite weakly connected digraph~$G$ with~${r \in V(G)}$ has ${k \in \mathbb{N}}$ edge-disjoint spanning pseudo-arborescences rooted in~$r$ 
    if and only if 
    for every bipartition~${(X, Y)}$ of~${V(G)}$ with~${r \in X}$ and finitely many edges between~$X$ and~$Y$ in either direction 
    there are at least~$k$ edges from~$X$ to~$Y$.
\end{theorem}

In fact we shall prove a slightly stronger version of this theorem, Theorem~\ref{thm:pseudo-edmonds}, which requires more notation. 

While minimal edges sets meeting every finite cut in an undirected multigraph turn out to be topological extensions of finite trees, there is no analogous topological interpretation of spanning pseudo-arborescences in terms of the Freudenthal compactification of the underlying multigraph. 
In Section~\ref{sec:pseudo-structure} we give an example of a digraph~$G$ with underlying multigraph~$H$ for which the closure in~$\abs{H}$ of the underlying undirected edges of any spanning pseudo-arborescence of~$G$ contains a homeomorphic image of~$S^1$. 
We shall be able to extend to pseudo-arborescences, in a suitable topological setting, the property of finite arborescences of being edge-minimal such that each vertex is still reachable by a directed path from the root. 
While in finite arborescences such directed paths are unique, however, their analogues in pseudo-arborescences are not in general unique.  
This will be illustrated by an example given in Section~\ref{sec:pseudo-structure}.

In order to formally encode reachability for infinite digraphs for our purpose, we introduce the notion of topological directed walks and paths via the Freudenthal compactification in Section~\ref{sec:top-dir}.
This section contains another main contribution of this paper, which is the extension of fundamental facts about the existence of directed walks and paths in finite digraphs to infinite ones using methods from topological infinite graph theory.

Finally, we prove the following theorem yielding a structural characterisation for spanning pseudo-arborescences. 

\begin{theorem}
    \label{thm:weak-r-reach_charaterisation}
    Let~$G$ be a locally finite weakly connected digraph and ${r \in V(G)}$. 
    Then the following statements are equivalent for an edge set ${F \subseteq E(G)}$ containing, 
    for every bipartition ${(X, Y)}$ of~${V(G)}$ with~${r \in X}$ and finitely many edges between~$X$ and $Y$ in either direction, 
    an edge from~$X$ to~$Y$. 
    \begin{enumerate}
        [label=(\roman*)]
        \item $F$ is a spanning pseudo-arborescence rooted in~$r$.
        \item For every vertex~${v \neq r}$ of~$G$ there is a unique edge in~$F$ whose head is~$v$, and no edge in~$F$ has~$r$ as its head. 
        \item For every weak component~$T$ of~${G[F]}$ the following holds: 
            If ${r \in V(T)}$, then~$T$ is an arborescence rooted in~$r$. 
            Otherwise, the underlying multigraph of~$T$ is a tree, $T$ contains a backwards directed ray and all other edges of~$T$ are directed away from that ray. 
    \end{enumerate}
\end{theorem}

We prove a slightly more general version of Theorem~\ref{thm:weak-r-reach_charaterisation} in Section~\ref{sec:pseudo-structure} (cf.~Theorem~\ref{thm:r-reach_charaterisation}).

The structure of this paper is as follows. 
In Section~\ref{sec:prelim} we give basic definitions and fix our notation for directed and undirected (multi-)graphs. 
We in particular refer to the topology we consider on locally finite (weakly) connected digraphs and (undirected) multigraphs, and state some basic lemmas that we shall need for our main results. 
In Section~\ref{sec:top-dir} we extend fundamental lemmas about directed walks and paths in finite digraphs to locally finite (weakly) connected digraphs. 
Section~\ref{sec:pseudo-packing} is dedicated to the proof of Theorem~\ref{thm:weak-pseudo-edmonds}. 
We complete the paper in Section~\ref{sec:pseudo-structure} with the proof of Theorem~\ref{thm:weak-r-reach_charaterisation} and a discussion about how much pseudo-arborescences resemble finite arborescences or topological trees.

\section{Preliminaries}
\label{sec:prelim}

For basic facts about finite and infinite graphs we refer the reader to \cite{diestel_buch}. 
As a source for facts about directed graphs we refer to \cite{bang-jensen}. 

Throughout the whole paper we shall often write ${G = (V, E)}$ for a digraph. 
Then~${V(G)}$ will denote its vertex set~$V$ and~${E(G)}$ its set of directed edges~$E$.
As for undirected graphs, we shall call the elements of~${E(G)}$ just edges. 
In general, we allow our digraphs to have parallel edges, but no loops. 
We view the edges of a digraph~$G$ as ordered pairs~${(a,b)}$ of vertices~${a, b \in V(G)}$ and shall write~${ab}$ instead of~${(a, b)}$, although this might not uniquely determine an edge. 
For an edge~${ab \in E(G)}$ we furthermore denote the vertex~$a$ as the \emph{tail} of~${ab}$ and~$b$ as the \emph{head} of~${ab}$. 

For two disjoint vertex sets ${X, Y}$ of a digraph~$G$ we denote by~${E(X, Y)}$ the set of all edges of~$G$ having their head in $X$ and their tail in $Y$ or their head in $Y$ and their tail in $X$.
By ${\overrightarrow{E}(X, Y)}$ we denote the set of edges of~$G$ that have their tail in~$X$ and their head in~$Y$. 
For a multigraph or digraph~$G$ we call the edge set~${E(X, Y)}$ a \emph{cut} if~${(X, Y)}$ is a bipartition of~${V(G)}$. 
If we introduce a cut~${E(X, Y)}$, then we implicitly want~${(X, Y)}$ to be the corresponding bipartition of~${V(G)}$ defining the cut. 
For a vertex set~${X \subseteq V(G)}$ we set~${d^+(X) := \abs{\overrightarrow{E}(X, V(G) \setminus X)}}$ and~${d^-(X) := \abs{\overrightarrow{E}(V(G) \setminus X, X)}}$. 
If~${X = \lbrace v \rbrace}$ for some vertex~${v \in V(G)}$, we write~${d^+(v)}$ instead of~${d^+(\lbrace v \rbrace)}$ and call it the \emph{out-degree} of~$v$. 
Similarly, we write~${d^-(v)}$ instead of~${d^-(\lbrace v \rbrace)}$ and call it the \emph{in-degree} of~$v$. 

For a finite non-trivial directed path~$P$ we call the vertex of out-degree~$1$ and in-degree~$0$ in~$P$ the \emph{start vertex of~$P$}. 
Similarly, we call the vertex of in-degree~$1$ and out-degree~$0$ in~$P$ the \emph{endvertex of~$P$}. 
If~$P$ consists only of a single vertex, we call that vertex the \emph{endvertex}~of~$P$. 

We define a \emph{finite directed walk} as a tuple~${(\mathcal{W}, <_{\mathcal{W}})}$ with the following properties: 
\begin{enumerate}
    \item $\mathcal{W}$ is a non-empty weakly connected graph with edge set ${\lbrace e_1, e_2, \ldots, e_n \rbrace}$ for some ${n \in \mathbb{N}}$ 
        such that the head of~$e_{i-1}$ is the tail of~$e_{i}$ for every ${i \in \mathbb{N}}$ satisfying~${2 \leq i \leq n}$. 
    \item $<_{\mathcal{W}}$ is a linear order on~${E(\mathcal{W})}$ stating that ${e_i <_{\mathcal{W}} e_j}$ if and only if ${i < j}$ for all~${i, j \in \lbrace 1, \ldots, n \rbrace}$. 
\end{enumerate}
Note that the second condition implies that the edges~${e_1, \ldots, e_n}$ are all distinct, i.e.~the walk traverses its edges only once. 
We call a directed walk without edges \emph{trivial} and call its unique vertex its \emph{endvertex}. 
Otherwise, we call the tail of~$e_1$ the \emph{start vertex} of~${(\mathcal{W}, <_{\mathcal{W}})}$ and the head of~$e_n$ the \emph{endvertex} of~${(\mathcal{W}, <_{\mathcal{W}})}$. 
If the start vertex and the endvertex of finite directed walk are equal, we call it \emph{closed}. 
Lastly, we call~${(\mathcal{W}, <_{\mathcal{W}})}$ a \emph{finite directed $s$--$t$~walk} for two vertices~${s, t \in V(\mathcal{W})}$ if~$s$ is the start vertex of~${(\mathcal{W}, <_{\mathcal{W}})}$ and~$t$ is the endvertex of~${(\mathcal{W}, <_{\mathcal{W}})}$. 
We might call a finite graph~$\mathcal{W}$ a finite directed walk and implicitly assume that there exists a linear order~${<_{\mathcal{W}}}$, which we then also fix, such that~${(\mathcal{W}, <_{\mathcal{W}})}$ is a finite directed walk. 
In particular, we will say that a finite directed walk~${(\mathcal{W}, <_{\mathcal{W}})}$ is contained in a graph~${G'}$ if~$\mathcal{W}$ is a subgraph of~${G'}$. 
Note that directed paths are directed walks when equipped with the obviously suitable linear order.

We define a \emph{ray} to be an undirected one-way infinite path. 
Any subgraph of a ray~$R$ that is itself a ray is called a \emph{tail} of~$R$. 

We call a weakly connected digraph~$R$ a \emph{backwards directed ray} if there is a unique vertex~${v \in V(R)}$ with~${d^-(v) = 1}$ and~${d^+(v) = 0}$ while~${d^-(w) = d^+(w) = 1}$ holds for every other vertex~${w \in V(R) \setminus \lbrace v \rbrace}$. 
A \emph{forwards directed ray} is analogously defined by interchanging~${d^-}$ and~${d^+}$. 

For an undirected multigraph~$G$ we define an equivalence relation on the set of all rays in~$G$. 
We call two rays in~$G$ \emph{equivalent} if they cannot be separated by finitely many vertices in~$G$. 
An equivalence class with respect to this relation is called an \emph{end} of~$G$. 
We denote the set of all ends of~$G$ by~${\Omega(G)}$. 
We define the \emph{ends of a digraph~$D$} precisely as the ends of its underlying multigraph. 
The set of all ends of~$D$ is also denoted by~${\Omega(D)}$. 
We say that a directed ray~$R$ of~$D$ is \emph{contained} in some end~${\omega \in \Omega(D)}$ if the underlying ray of~$R$ is contained in the end~$\omega$ of the underlying multigraph of~$D$.

We call a digraph~$A$ an \emph{out-arborescence rooted in~$r$} if~${r \in V(A) \cup \Omega(A)}$ and the underlying multigraph of~$A$ is a tree such that~${d^-(v) = 1}$ holds for every vertex~${v \in V(A) \setminus \lbrace r \rbrace}$ and additionally~${d^-(r) = 0}$ in the case that~${r \in V(A)}$, 
while we demand that~$r$ contains a backwards directed ray if~${r \in \Omega(A)}$. 

Note that if~${r \in V(A)}$, then~$A$ does not contain a backwards directed ray. 
In the case where~${r \in \Omega(A)}$, then~$r$ is the unique end of~$A$ containing a backwards directed ray, since a second one would yield a vertex with in-degree bigger than~$1$ by using that the underlying multigraph of~$A$ is a tree. 
Also note that if~$A$ is a finite digraph, the condition~${d^-(r) = 0}$ for~${r \in V(A)}$ in the definition of an out-arborescence rooted in~$r$ is redundant, because it is implied by the tree structure of~$A$. 

Similarly, an \emph{in-arborescence rooted in~$r$} is defined with~${d^-}$ replaced by~${d^+}$. 
Corresponding results about in-arborescences are immediate by reversing the orientations of all edges. 
For both types of arborescences we call~$r$ the \emph{root} of the arborescence. 
In this paper we shall only work with out-arborescences. 
Hence, we shall drop the prefix `out' and just write arborescence from now on. 

A multigraph is called \emph{locally finite} if each vertex has finite degree. 
We further call a digraph \emph{locally finite} if its underlying multigraph is locally finite. 

For a vertex set~$X$ in a locally finite connected multigraph~$G$ we define its \emph{combinatorial closure}~${\overline{X} \subseteq V(G) \cup \Omega(G)}$ as the set~$X$ together with all ends of~$G$ that contain a ray which we cannot separate from~$X$ by finitely many vertices. 
Note that for a finite cut~${E(X, Y)}$ of~$G$ we obtain that~${(\overline{X}, \overline{Y})}$ is a bipartition of~${V(G) \cup \Omega(G)}$, 
because every end in~$\overline{X}$ can be separated from~$Y$ by the finitely many vertices of~$X$ that are incident with edges of~${E(X, Y)}$, 
and, furthermore, each ray contains a subray that is either completely contained in~$X$ or in~$Y$ since~${E(X, Y)}$ is finite. 
The \emph{combinatorial closure} of a vertex set in a digraph is just defined as the combinatorial closure of that set in the underlying undirected multigraph.

Let~$G$ be a locally finite digraph and~${Z \subseteq V(G) \setminus \lbrace r \rbrace}$ with~${r \in V(G) \cup \Omega(G)}$. 
An edge set~${F \subseteq E(G)}$ is called \emph{$r$-reachable for~$Z$} if~${\abs{F \cap \overrightarrow{E}(X, Y)} \geq 1}$ holds for every finite cut~${E(X, Y)}$ of~$G$ with~${r \in \overline{X}}$ and~${Y \cap Z \neq \emptyset}$.
Furthermore, if~$F$ is an $r$-reachable set for~${V(G) \setminus \lbrace r \rbrace}$, we call~$F$ a \emph{spanning $r$-reachable set}. 
Note that a spanning $r$-reachable set spans~${V(G)}$ as an edge set. 
We continue with a very basic remark about spanning $r$-reachable sets.

\begin{remark}
    \label{rem:basic_rem}
    Let~$G$ be a locally finite digraph with a spanning $r$-reachable set~$F$ with~${r \in V(G) \cup \Omega(G)}$. 
    Then ${\abs{F \cap \overrightarrow{E}(V(G) \setminus M, M)} \geq 1}$ holds for every non-empty finite set~${M \subseteq V(G)}$ with~${r \notin M}$.
\end{remark}

\begin{proof}
    Since~$G$ is locally finite and~$M$ is finite, we know that the cut~${E(V(G) \setminus M, M)}$ is finite. 
    The assumption~${r \notin M}$ ensures that~${r \in \overline{V(G) \setminus M}}$. 
    Using that~$F$ is a spanning \linebreak $r$-reachable set and that~$M$, as a non-empty set, contains a vertex different from~$r$, we get the desired inequality~${\abs{F \cap \overrightarrow{E}(V(G) \setminus M, M)} \geq 1}$ by the definition of spanning $r$-reachable sets. 
\end{proof}

Note that for a locally finite digraph~$G$ with a spanning $r$-reachable set~$F$ the digraph~${G[F]}$ is spanning. 
This follows by applying Remark~\ref{rem:basic_rem} to the set~${M := \lbrace v \rbrace}$ for every vertex~${v \in V(G)}$. 
Furthermore, note that if~$G$ is finite, the subgraph induced by a spanning $r$-reachable set contains a spanning arborescence rooted in~${r \in V(G)}$. 

We conclude this section with a last definition. 
We call an inclusion-wise minimal \linebreak $r$-reachable set~$F$ for a set~${Z \subseteq V(G) \setminus \lbrace r \rbrace}$ a \emph{pseudo-arborescence for~$Z$ rooted in~$r$}. 
Moreover, if~$F$ is spanning, i.e.~${Z = V(G) \setminus \lbrace r \rbrace}$, we call it a \emph{spanning pseudo-arborescence rooted in~$r$}.

\subsection{Topological notions for undirected multigraphs}

Throughout this subsection let~${G = (V, E)}$ denote a locally finite connected multigraph. 
We can endow~$G$ together with its ends with a topology which yields the topological space~$\abs{G}$. 
A precise definition of~$\abs{G}$ for locally finite connected simple graphs can be found in \cite{diestel_buch}*{Chapter~8.5}. 
However, this concept and definition directly extends to locally finite connected multigraphs.  
For a better understanding we should point out here that a ray of~$G$ converges in~$\abs{G}$ to the end of~$G$ that it is contained in. 
An equivalent way of describing~$\abs{G}$ is by first endowing~$G$ with the topology of a $1$-complex and then compactifying this space using the Freudenthal compactification~\cite{freud-equi}. 

For an edge~${e \in E}$ let~$\mathring{e}$ denote the set of points in~$\abs{G}$ that correspond to inner points of the edge~$e$. 
For an edge set~${F \subseteq E}$ we define~${\mathring{F} = \bigcup \lbrace \mathring{e} \; ; \; e \in F \rbrace \subseteq \abs{G}}$. 
Given a point set~$X$ in~$\abs{G}$, we denote the \emph{closure} of~$X$ in~$\abs{G}$ by~$\overline{X}$. 
To ease notation we shall also use this notation when~$X$ denotes an edge set or a subgraph of~$G$, meaning that we apply the closure operator to the set of all points in~$\abs{G}$ that correspond to~$X$. 
Note that for a vertex set its closure coincides with its combinatorial closure in locally finite connected multigraphs. 
Hence, we shall use the same notation for these two operators. 
Furthermore we call a subspace~${Z \subseteq \abs{G}}$ \emph{standard} if~${Z = \overline{H}}$ for some subgraph~$H$ of~$G$. 

Let~${W \subseteq \abs{G}}$ and~$<_W$ be a linear order on~${\mathring{E} \cap W}$. 
We call the tuple~${(W, <_W)}$ a \emph{topological walk} in~$\abs{G}$ if there exists a continuous map~${\sigma : [0,1] \longrightarrow |G|}$ such that the following hold:

\begin{enumerate}
    \item $W$ is the image of~$\sigma$, 
    \item each point~${p \in \mathring{E} \cap W}$ has precisely one preimage under~$\sigma$, and
    \item the linear order~$<_W$ equals the linear order~$<_{\sigma}$ on~${\mathring{E} \cap W}$ 
        defined via ${p <_{\sigma} q}$ if and only if ${\sigma^{-1}(p) <_{\mathbb{R}} \sigma^{-1}(q)}$, where~$<_{\mathbb{R}}$ denotes the natural linear order of the reals. 
\end{enumerate}

We call such a map~$\sigma$ a \emph{witness} of~${(W, <_W)}$. 
When we talk about a topological walk~${(W, <_W)}$ we shall often omit stating its linear order~$<_W$ explicitly and just refer to the topological walk by writing~$W$. 
In particular, we might say that a topological walk~${(W, <_W)}$ is contained in some subspace~$X$ of~$\abs{G}$ if~${W \subseteq X}$ holds. 
Furthermore, we call a point~$x$ of~$\abs{G}$ an \emph{endpoint} of~$W$ if~$0$ or~$1$ is mapped to~$x$ by a witness of~$W$.  
Note that this definition is independent of the particular witness. 
Similar to finite walks in graphs we call an endpoint~$x$ of~$W$ an \emph{endvertex} of~$W$ if~$x$ corresponds to a vertex of~$G$. 
Furthermore, we denote~$W$ as an $x$--$y$ topological walk, if~$x$ and~$y$ are endpoints of~$W$.
If~$W$ has just one endpoint, which then has to be an end or a vertex by definition, we call it \emph{closed}. 
Note that an $x$--$y$ topological walk is a standard subspace for any~${x, y \in V \cup \Omega(G)}$.
We say that a witness~$\sigma$ of a topological walk~$W$ \emph{pauses} at a vertex~${v \in V}$ if the preimage of~$v$ under~$\sigma$ is a disjoint union of closed nontrivial intervals. 

We define an \emph{arc} in~$\abs{G}$ as the image of a homeomorphism mapping into~$\abs{G}$ and with the closed real unit interval~${[0, 1] \subseteq \mathbb{R}}$ as its domain. 
Note that arcs in~$\abs{G}$ are also topological walks in~$\abs{G}$ if we equip them with a suitable linear order, of which there exist only two. 
Since the choice of such a linear order does not change the set of endpoints of the arc if we then consider it as a topological walk, we shall use the notion of endpoints and endvertices also for arcs. 
Furthermore, note that finite paths of~$G$ which contain at least one edge correspond to arcs in~$\abs{G}$, but again there might be infinite subgraphs, for example rays, whose closures form arcs in~$\abs{G}$. 
We now call a subspace~$X$ of~$\abs{G}$ \emph{arc-connected} if there exists an $x$--$y$~arc in~$X$ for any two points~${x, y \in X}$. 

Lastly, we define a \emph{circle} in~$\abs{G}$ as the image of a homeomorphism mapping into~$\abs{G}$ and with the unit circle~${S^1 \subseteq \mathbb{R}^2}$ as its domain. 
We might also consider any circle as a closed topological walk if we equip it with a suitable linear order, which, however, depends on the point on the circle that we choose as the endpoint for the closed topological walk, and on choosing one of the two possible orientations of~$S^1$. 
Similarly as for finite paths, note that finite cycles in~$G$ correspond to circles in~$\abs{G}$, but there might be infinite subgraphs of~$G$ whose closures are circles in~$\abs{G}$ as well. 

Using these definitions we can now formulate a topological extension of the notion of trees. 
We define a \emph{topological tree} in~$\abs{G}$ as an arc-connected standard subspace of~$\abs{G}$ that does not contain any circle. 
Note that in a topological tree there is a unique arc between any two points of the topological tree, which resembles a property of finite trees with respect to vertices and finite paths. 
Furthermore, we denote by a \emph{topological spanning tree of~$G$} a topological tree in~$\abs{G}$ that contains all vertices of~$G$. 
Since topological spanning trees are closed subspaces of~$\abs{G}$, they need to contain all ends of~$G$ as well.

\subsection{Topological notions for digraphs}

In this subsection we extend some of the notions of the previous subsection to directed graphs. 
Throughout this subsection let~$G$ denote a locally finite weakly connected digraph and let~$H$ denote its underlying multigraph. 
We define the topological space~$\abs{G}$ as~$\abs{H}$. 
Additionally, every edge~${e = uv \in E(G)}$ defines a certain linear order~$<_e$ on ${\overline{\lbrace e \rbrace} \subseteq \abs{G}}$ via its direction. 
For the definition of~$<_e$ we first take any homeomorphism ${\varphi_e: [0, 1] \longrightarrow \overline{\lbrace e \rbrace} \subseteq \abs{G}}$ with~${\varphi_e(0) = u}$ and~${\varphi_e(1) = v}$. 
Now we set~${p <_e q}$ for arbitrary~${p, q \in \overline{\lbrace e \rbrace}}$ if ${\varphi_e^{-1}(p) <_{\mathbb{R}} \varphi_e^{-1}(q)}$ where~${<_{\mathbb{R}}}$ is the natural linear order on the real numbers. 
Note that the definition of~$<_e$ does not depend on the choice of the homeomorphism~$\varphi_e$. 

Let~${(W, <_W)}$ be a topological walk in~$\abs{G}$ with witness~$\sigma$. 
We call~${(W, <_W)}$ \emph{directed} if~${<_e {\upharpoonright} \mathring{e}}$ equals~${<_W {\upharpoonright} \mathring{e}}$ for every edge~${e \in E(G)}$ with~${\mathring{e} \cap W \neq \emptyset}$. 
If~${(W, <_W)}$ is directed and~${\sigma(0) = s \neq t = \sigma (1)}$ for~${s, t \in \abs{G}}$, then there is no linear order~$<'_W$ such that~${(W, <'_W)}$ is a directed topological walk with a witness~$\sigma'$ satisfying~${\sigma'(0) = t}$ and~${\sigma'(1) = s}$, because every topological $s$--$t$~walk uses inner points of some edge. 
Hence, if we consider a directed topological $s$--$t$~walk~${(W, <_W)}$ for~${s, t \in \abs{G}}$, we implicitly assume that~${\sigma(0) = s \neq t = \sigma (1)}$ holds for every witness~$\sigma$ of~${(W, <_W)}$. 

As arcs and circles can be seen as special instances of topological walks, \emph{directed arcs} and \emph{directed circles} are analogously defined. 
Note that if we can equip an arc with a suitable linear order such that it becomes a directed topological walk, then this linear order is unique. 
Hence, when we call an arc directed we implicitly associate this unique linear order with it.

\subsection{Basic lemmas}

The proofs of two lemmas (Lemma~\ref{lem:equi_lemma_walk} and Lemma~\ref{lem:k_many-reachable}) rely to some extent on compactness arguments. 
At those points it will be sufficient for us to use the following lemma, which is known as K\"onig's Infinity Lemma. 

\begin{lemma}
    \label{lem:koenig}
    \cite{diestel_buch}*{Lemma~8.1.2}
    Let ${(V_i)_{i \in \mathbb{N}}}$ be a sequence of disjoint non-empty finite sets, and let~$G$ be a graph on their union. 
    Assume that for every~${n > 0}$ each vertex in~$V_n$ has a neighbour in~$V_{n-1}$. 
    Then~$G$ contains a ray ${v_0v_1 \ldots}$ with~${v_n \in V_n}$ for all~${n \in \mathbb{N}}$. 
\end{lemma}

We shall heavily work with the topological space~$\abs{G}$ of a locally finite multigraph~$G$ appearing as the underlying graph of digraphs we consider. 
Therefore, we shall make use of some basic statements and properties of the space~$\abs{G}$, in particular those involving connectivity. 
Although the following lemmas are only stated for locally finite graphs, their proofs immediately extend to locally finite multigraphs. 

\begin{proposition}
    \label{prop:compact_Hausdorff}
    \cite{diestel_buch}*{Lemma~8.5.1}
    If~$G$ is a locally finite connected multigraph, then~$\abs{G}$ is a compact Hausdorff space. 
\end{proposition}

The next lemma is essential for decoding the topological property of arc-connectedness of standard subspaces of~$\abs{G}$ into a combinatorial one. 

\begin{lemma}
    \label{lem:jumping-arc}
    \cite{diestel_buch}*{Lemma~8.5.3}
    Let~$G$ be a locally finite connected multigraph and ${F \subseteq E(G)}$ be a cut with sides~$V_1$ and~$V_2$. 
    \begin{enumerate}
        [label=(\roman*)]
        \item {If~$F$ is finite, then~${\overline{V_1} \cap \overline{V_2} = \emptyset}$, and there is no arc in~${\abs{G} \setminus \mathring{F}}$ with one endpoint in~$V_1$ and the other in~$V_2$.} 
        \item {If~$F$ is infinite, then~${\overline{V_1} \cap \overline{V_2} \neq \emptyset}$, and there may be such an arc.} 
    \end{enumerate}
\end{lemma}

Note that for a finite cut~${E(X, Y)}$ of~$G$ we obtain that~${(\overline{X}, \overline{Y})}$ bipartitions~${V(G) \cup \Omega(G)}$.

The following lemma captures the equivalence of arc-connectedness and connectedness for standard subspaces of~$\abs{G}$. 

\begin{lemma}
    \label{lem:equi_con_arc}
    \cite{diestel_buch}*{Lemma~8.5.4}
    If~$G$ is a locally finite connected multigraph, then every connected standard subspace of~$\abs{G}$ is arc-connected.
\end{lemma}

We conclude with a convenient lemma which combines the essences of the previous two. 

\begin{lemma}
    \label{lem:equi_con_each_cut}
    \cite{diestel_buch}*{Lemma~8.5.5}
    If~$G$ is a locally finite connected multigraph, then a standard subspace of~$\abs{G}$ is connected if and only if it contains an edge from every finite cut of~$G$ of which it meets both sides. 
\end{lemma}

\section{Fundamental statements about topological directed walks in locally finite digraphs}
\label{sec:top-dir}

In this section we lift several facts about topological walks and arcs to their directed counterparts. 
Most of the involved techniques and proof ideas are similar to the ones used in undirected locally finite connected multigraphs. 
Nevertheless, because of the overlying directed structure on the multigraph, some adjustments and additional arguments are needed in the proofs. 
We start with a statement that combinatorially characterises the existence of directed topological walks in a standard subspace via finite cuts.

\begin{lemma}
    \label{lem:equi_lemma_walk}
    Let~$G$ be a locally finite weakly connected digraph,~${s, t \in V(G) \cup \Omega(G)}$ with~${s \neq t}$ and~${F \subseteq E(G)}$. 
    Then the following statements are equivalent. 
    \begin{enumerate}
        [label=(\roman*)]
        \item \label{item:walk1}
            $\overline{F}$ contains a directed topological $s$--$t$~walk.
        \item \label{item:walk2}
            ${\abs{F \cap \overrightarrow{E}(X, Y)} \geq 1}$ for every finite cut~${E(X, Y)}$ of~$G$ with~${s \in \overline{X}}$ and~${t \in \overline{Y}}$. 
        \item \label{item:walk3} 
            There is a subset ${W \subseteq F}$ such that 
            ${\abs{W \cap \overrightarrow{E}(X, Y)} = \abs{W \cap \overrightarrow{E}(Y, X)} + 1}$ for every finite cut~${E(X, Y)}$ of~$G$ with~${s \in \overline{X}}$ and~${t \in \overline{Y}}$. 
    \end{enumerate}
\end{lemma}

\begin{proof}
    First we prove the implication from~\ref{item:walk1} to~\ref{item:walk3}. 
    Let~${E(X, Y)}$ be any finite cut of~$G$ with~${s \in \overline{X}}$ and~${t \in \overline{Y}}$. 
    Since~$\overline{F}$ contains a directed topological $s$--$t$~walk~${(\overline{W}, <_{\overline{W}})}$ for an edge set~${W \subseteq E(G)}$, we know that~${F \cap E(X, Y) \neq \emptyset}$ by Lemma~\ref{lem:equi_con_each_cut}. 
    Note furthermore that~${\overline{X} \cap \overline{Y} = \emptyset}$ by Lemma~\ref{lem:jumping-arc}. 
    As~$\overline{X}$ and~$\overline{Y}$ are closed and~$\abs{G}$ is compact by Proposition~\ref{prop:compact_Hausdorff}, we get that~$\overline{X}$ and~$\overline{Y}$ are compact too. 
    Now let~$\varphi$ be a witness of~$\overline{W}$. 
    Since~$\overline{Y}$ is compact and~$\varphi$ is continuous, there exists a smallest number~${q \in [0, 1]}$ such that~${\varphi(q) \in \overline{Y}}$. 
    Furthermore, there is a biggest number~${p \in [0, q]}$ such that~${\varphi(p) \in \overline{X}}$. 
    Note that~${p \neq q}$ since~${\overline{X} \cap \overline{Y} = \emptyset}$. 
    Now let~${M := \lbrace \varphi(r) \in \abs{G} \; ; \; p < r < q \rbrace}$. 
    Obviously,~$M$ contains only inner points of edges in~${E(X, Y)}$. 
    Since~$M$ is connected, we obtain~${M = \mathring{e}}$ for some edge~${e \in E(X, Y)}$. 
    Using that~${<_{\overline{W}} {\upharpoonright} \mathring{e}}$ equals~${<_e {\upharpoonright} \mathring{e}}$ because~${(\overline{W}, <_{\overline{W}})}$ is a directed $s$--$t$~walk, we see that~${e \in W \cap \overrightarrow{E}(X, Y)}$. 
    By the continuity of~$\varphi$, we get that ${\varphi(q) = y}$ for some vertex~${y \in Y}$. 
    If ${\abs{W \cap \overrightarrow{E}(Y, X)} = 0}$, we know that~${\abs{W \cap \overrightarrow{E}(X, Y)} = 1}$ since ${\varphi {\upharpoonright} [q, 1]}$ is connected and hence a subset of~$\overline{Y}$, and we are done.
    Otherwise, consider ${\varphi {\upharpoonright} [q, 1]}$, which is a witness for~${(\overline{Q}, <_{\overline{Q}})}$ being a directed $y$--$t$ walk where 
    ${Q = \{ e \in W \; ; \; \forall a \in \mathring{e} \; : \; \varphi^{-1}(a) > q \}}$. 
    Note that since ${\abs{W \cap \overrightarrow{E}(Y, X)} > 0}$ we get that ${\abs{Q \cap \overrightarrow{E}(Y, X)} > 0}$ as well by the choice of~$q$.
    Therefore, $\overline{Q}$ also contains an element of~$\overline{X}$.
    Similarly as before, let ${p' \in [q, 1]}$ denote the smallest number such that~${\varphi(p') \in \overline{X}}$ and ${q' \in [q, p']}$ denote the biggest number such that~${\varphi(q') \in \overline{Y}}$.
    Now considering the set ${M' := \lbrace \varphi(r) \in \abs{G} \; ; \; q' < r < p' \rbrace}$ we obtain as before that ${M' = \mathring{f}}$ for some edge~${f \in E(X, Y)}$.
    More precisely, since~$Q$ is a directed $x$--$t$ walk, we get that ${f \in {Q \cap \overrightarrow{E}(Y, X)}}$.
    Finally, we consider the directed $\varphi(p')$--$t$ walk ${(\overline{P}, <_{\overline{P}})}$ with witness ${\varphi {\upharpoonright} [p', 1]}$ where ${P = \{ e \in W \; ; \; \forall a \in \mathring{e} \; : \; \varphi^{-1}(a) > p' \}}$.
    By the previous observations we know that ${\abs{P \cap \overrightarrow{E}(X, Y)} = \abs{W \cap \overrightarrow{E}(X, Y)} - 1}$ and ${\abs{P \cap \overrightarrow{E}(Y, X)} = \abs{W \cap \overrightarrow{E}(Y, X)} - 1}$ hold.
    Using that~${E(X, Y)}$ contains only finitely many edges, we inductively get that~${\abs{W \cap \overrightarrow{E}(X, Y)} = \abs{W \cap \overrightarrow{E}(Y, X)} + 1}$ is true.

    The implication from~\ref{item:walk3} to~\ref{item:walk2} is immediate. 
    
    It remains to show that~\ref{item:walk2} implies~\ref{item:walk1}. 
    For this we first fix a sequence~${(S_n)_{n \in \mathbb{N}}}$ of finite vertex sets~${S_n \subseteq V(G)}$ such that~${S_n \subsetneqq S_{n+1}}$ for every~${n \in \mathbb{N}}$ and~${\bigcup_{n \in \mathbb{N}} S_n = V(G)}$. 
    For every~${n \in \mathbb{N}}$ let~$G_n$ denote the digraph which arises by contracting~${E(G-S_n)}$ in~$G$. 
    Since~$G$ is locally finite, we know that each~$G_n$ is a finite digraph. 
    We call the vertices of~$G_n$ that are not contained in~$S_n$ \emph{dummy vertices}. 
    Note that each dummy vertex of~$G_n$ corresponds to a unique weak component of~${G - S_n}$. 
    
    If some~${v \in V(G) \cup \Omega(G)}$ is not contained in~$S_n$, there exists a unique component~$C_n$ of~${G-S_n}$ such that~${v \in \overline{C_n}}$. 
    This is obviously true if~$v$ is a vertex of~$G$, but also holds if~$v$ is an end of~$G$. 
    To see the latter statement suppose~${v \in \Omega(G)}$ is contained in~$\overline{C_n}$ for a component~$C_n$ of~$G-S_n$. 
    Then the cut~${E(V(C_n), V(G) \setminus V(C_n))}$ is finite as~$S_n$ is finite and~$G$ is locally finite. 
    Hence ${\overline{V(C_n)} \cap (\overline{V(G) \setminus V(C_n))} = \emptyset}$ by Lemma~\ref{lem:jumping-arc}, which means that~$v$ cannot lie in the closure of another component of~${G - S_n}$. 
    With a slight abuse of notation, we refer to the dummy vertex of~$G_n$ corresponding to~$C_n$ as~$v$.

    Since for each~${n \in \mathbb{N}}$ every cut of~$G_n$ corresponds to a finite cut of~$G$, we obtain by Theorem~\ref{thm:fin-edmonds} that~${F \cap E(G_n)}$ contains the edge set of a finite directed $s$--$t$~walk in the digraph~$G_n$. 
    Furthermore, any finite directed $s$--$t$~walk~${(\mathcal{W}_{n+1}, <_{\mathcal{W}_{n+1}})}$ in~$G_{n+1}$ induces a finite directed $s$--$t$~walk~${(\mathcal{W}_{n}, <_{\mathcal{W}_{n}})}$ in~$G_n$ 
    via~${E(\mathcal{W}_{n}) := E(\mathcal{W}_{n+1}) \cap E(G_n)}$ and defining~${<_{\mathcal{W}_{n}}}$ as~${<_{\mathcal{W}_{n+1}} {\upharpoonright} E(\mathcal{W}_{n})}$. 
    Note that each maximal interval with respect to~${<_{\mathcal{W}_{n+1}}}$ of~${E(\mathcal{W}_{n+1}) \setminus E(\mathcal{W}_n)}$ corresponds to some $v$--$w$~walk where~$v$ and~$w$ are the same dummy vertex of~$G_n$. 
    Hence each time a dummy vertex of~$G_n$ appears as the head of some edge~${e \in E(\mathcal{W}_n)}$ there is a corresponding, possibly trivial, walk~$\mathcal{W}_{n+1}^e$ using edges of of such a maximal interval with the induced order~${<_{\mathcal{W}_{n+1}}{\upharpoonright}E(\mathcal{W}_{n+1}^e)}$. 
    
    For every~${n \in \mathbb{N}}$ let~$V_n$ denote the set of all finite directed $s$--$t$~walks in~$G_n$ that use only edges from~$F$. 
    Obviously, each set~$V_n$ is finite as~$G_n$ is a finite digraph. 
    By the previously given arguments, none of the sets~$V_n$ is empty and each element of~$V_{n+1}$ induces one of~$V_n$. 
    Hence, we get a sequence~${((\mathcal{W}_{n}, <_{\mathcal{W}_{n}}))_{n \in \mathbb{N}}}$ of finite directed $s$--$t$~walks with~${(\mathcal{W}_{n}, <_{\mathcal{W}_{n}}) \in V_n}$ such that~${{E(\mathcal{W}_{n+1}) \cap E(\mathcal{W}_{n}) = E(\mathcal{W}_{n})}}$ and~${<_{\mathcal{W}_{n+1}} {\upharpoonright} E(\mathcal{W}_{n})}$ equals~${<_{\mathcal{W}_{n}}}$ for every~${n \in \mathbb{N}}$ by Lemma~\ref{lem:koenig}. 
    We define~${W_n := E(\mathcal{W}_{n})}$ for every~${n \in \mathbb{N}}$. 
    Next we set~${W := \bigcup_{n \in \mathbb{N}} W_{n}}$ and~${<_W := \bigcup_{n \in \mathbb{N}} <_{\mathcal{W}_{n}}}$. 
    Furthermore, we define a linear order~${<_{\overline{W}}}$ on~$\mathring{W}$ as follows for~${p, q \in \mathring{W}}$ with~${p \neq q}$: 
    \[
        p <_{\overline{W}} q \; \textnormal{ if } \; 
            \begin{cases} 
                p \in \mathring{e} \textnormal{ and } q \in \mathring{f} \textnormal{  with } e <_W f \textnormal{ for some } e,f \in W \textnormal{ with } e \neq f \textnormal{, or} \\
                p, q \in \mathring{e} \textnormal{  and } p <_e q \textnormal{ for some } e \in W.
            \end{cases}
    \]
    Now we claim that~${(\overline{W}, <_{\overline{W}})}$ is a directed topological $s$--$t$~walk in~$\abs{G}$. 
    In order to show this we first have to define a witness~$\varphi$ for~${(\overline{W}, <_{\overline{W}})}$. 
    We shall obtain~$\varphi$ as a limit of countably many certain witnesses~$\varphi_n$ of directed topological walks~${(\overline{W_n}, <_{\overline{W_n}})}$ in~$\abs{G_n}$ that we define inductively, where~${<_{\overline{W_n}}}$ is defined analogously as~${<_{\overline{W}}}$ but with respect to~$W_n$. 
    
    For~${n = 0}$ we start with a witness~$\varphi_0$ of the directed topological $s$--$t$~walk~${(\overline{W_0}, <_{\overline{W_0}})}$ in~$\abs{G_0}$ which pauses at every dummy vertex of~$G_0$ contained in~$\overline{W_0}$. 
    
    Now suppose that the witness~$\varphi_n$ of~${(\overline{W_n}, <_{\overline{W_n}})}$ has already been defined such that it pauses at every dummy vertex of~$G_n$ that is contained in~$\overline{W_n}$. 
    Then we define~$\varphi_{n+1}$ as some witness of~${(\overline{W_{n+1}}, <_{\overline{W_{n+1}}})}$ as follows. 
    For every edge~${e \in W_n}$ whose head is a dummy vertex of~$G_n$, let~$W_{n+1}^e$ be the edge set of the walk~$\mathcal{W}_{n+1}^e$ as above and let~$\varphi_{n+1}^e$ be a witness that~$\overline{\mathcal{W}_{n+1}^e}$ is the corresponding directed topological walk that pauses at every dummy vertex of~$G_{n+1}$ that is contained in~$\overline{W_{n+1}^e}$. 
    Starting with~$\varphi_{n}$, each time we enter some dummy vertex~$d$ of~$G_n$ by an edge~$e$, we replace the image of the interval that is mapped to~$d$ with a rescaled version of~$\varphi_{n+1}^e$. 
    
    Using the maps~$\varphi_n$ we are able to define~$\varphi$ as follows: 
    For every~${q \in [0, 1]}$ for which there exists an~${n \in \mathbb{N}}$ such that~${\varphi_n(q) \in |G[S_n]| \subseteq |G_n|}$, we set~${\varphi(q) := \varphi_n(q)}$. 
    Otherwise,~$\varphi_n(q)$ corresponds to a contracted component~$C_n$ of~${G - S_n}$ for every~${n \in \mathbb{N}}$. 
    Since~${S_n \subsetneqq S_{n+1}}$ for every~${n \in \mathbb{N}}$ and~${\bigcup_{n \in \mathbb{N}} S_n = V(G)}$, it is easy to check that~${\bigcap_{n \in \mathbb{N}} \overline{C_n} = \{ \omega \}}$ for some end~$\omega$ of~$G$. 
    In this case, we define~${\varphi(q) := \omega}$. 
    This completes the definition of~$\varphi$. 
    It is straightforward to verify that~$\varphi$ is continuous and also onto~$\overline{W}$ because each~$\varphi_n$ is onto~$\overline{W_n}$ and~${W := \bigcup_{n \in \mathbb{N}} W_n}$. 
    This ensures that it is a witness of~${(\overline{W}, <_{\overline{W}})}$ being a topological \linebreak $s$--$t$~walk. 
    Note that the linear order~${<_{\overline{W}} {\upharpoonright} \mathring{e}}$ equals~${<_{e} {\upharpoonright} \mathring{e}}$ for each edge~${e \in W}$ since each linear order~$<_{\overline{W_n}}$ has this property. 
    Hence,~$\varphi$ witnesses that~${(\overline{W}, <_{\overline{W}})}$ is a directed topological $s$--$t$~walk in~$\abs{G}$ with~${W \subseteq F}$. 
\end{proof}

We proceed with a lemma which gives a combinatorial description for a standard subspace to be a directed arc. 

\begin{lemma}
    \label{lem:equi_lemma_arc}
    Let~$G$ be a locally finite weakly connected digraph, ${s, t \in V(G) \cup \Omega(G)}$ with~${s \neq t}$ and~${A \subseteq E(G)}$. 
    Then the following statements are equivalent:
    \begin{enumerate}
        [label=(\roman*)]
        \item \label{item:arc1}
            $\overline{A}$ is a directed $s$--$t$~arc. 
        \item \label{item:arc2}
            $A$ is inclusion-wise minimal such that~${\abs{A \cap \overrightarrow{E}(X, Y)} \geq 1}$ holds for every finite cut~${E(X, Y)}$ of~$G$ with~${s \in \overline{X}}$ and~${t \in \overline{Y}}$. 
        \item \label{item:arc3}
            $A$ is inclusion-wise minimal such that ${\abs{A \cap \overrightarrow{E}(X, Y)} = \abs{A \cap \overrightarrow{E}(Y, X)} + 1}$ holds for every finite cut~${E(X, Y)}$ of~$G$ with~${s \in \overline{X}}$ and ${t \in \overline{Y}}$. 
    \end{enumerate}
\end{lemma}

\begin{proof}
    First we show the implication from~\ref{item:arc1} to~\ref{item:arc3}. 
    As~$\overline{A}$ is a directed $s$--$t$~arc, it is also a directed topological $s$--$t$~walk.
    So by Lemma~\ref{lem:equi_lemma_walk}, we only need to check the minimality of~$A$ for property~\ref{item:arc3}. 
    Since~$\overline{A}$ is an $s$--$t$~arc, we know that~$s$ and~$t$ are in different topological components of~${\overline{A \setminus \lbrace e \rbrace}}$ for any edge~${e \in A}$. 
    So no proper subset of~$A$ has the property that its closure in~$\abs{G}$ contains a directed topological $s$--$t$~walk. 
    Again by Lemma~\ref{lem:equi_lemma_walk} we know that no proper subset of~$A$ satisfies statement~\ref{item:walk3} of Lemma~\ref{lem:equi_lemma_walk}. 
    This proves the minimality of~$A$ and hence statement~\ref{item:arc3}. 
    
    Next let us verify that~\ref{item:arc3} implies~\ref{item:arc2}. 
    Assume for a contradiction that statement~\ref{item:arc3} holds, but~\ref{item:arc2} does not. 
    Then there must exist a proper subset~${A' \subsetneqq A}$ 
    meeting~${\overrightarrow{E}(X, Y)}$ for every finite cut~${E(X, Y)}$ of~$G$ with~${s \in  \overline{X}}$ and~${t \in \overline{Y}}$. 
    By Lemma~\ref{lem:equi_lemma_walk} we get that~$A'$ satisfies also statement~\ref{item:walk3} of Lemma~\ref{lem:equi_lemma_walk}. 
    This contradicts the minimality of~$A$. 
    
    It remains to prove the implication from~\ref{item:arc2} to~\ref{item:arc1}. 
    By assuming~\ref{item:arc2} we know from Lemma~\ref{lem:equi_lemma_walk} that~$\overline{A}$ contains a directed topological $s$--$t$~walk and by the minimality of~$A$ we know that~$\overline{A}$ is in fact a directed topological $s$--$t$~walk, say witnessed by~${\varphi : [0, 1] \longrightarrow \abs{G}}$. 
    Now suppose for a contradiction that~$\overline{A}$ is not a directed $s$--$t$~arc. 
    Then there exists a point~${a \in V(G) \cup \Omega(G)}$ that spoils injectivity for~$\varphi$. 
    Note that~$\overline{A}$ is compact because it is a closed set in~$\abs{G}$, which is a compact space by Proposition~\ref{prop:compact_Hausdorff}. 
    Since~$\varphi$ is continuous and~$\overline{A}$ is compact, there exists a smallest number~${x \in [0, 1]}$ and a largest number~${y \in [0, 1]}$ such that~${\varphi(x) = \varphi(y) = a}$. 
    We obtain from this that the image of~${\varphi {\upharpoonright} [0, x]}$ is a directed topological $s$--$a$~walk and the image of~${\varphi {\upharpoonright} [y, 1]}$ is a directed topological $a$--$t$~walk. 
    Concatenating these two walks yields another directed topological $s$--$t$~walk, which is the closure in~$\abs{G}$ of some edge set~${A' \subseteq A}$. 
    Knowing that~${x \neq y}$, we get that~${A' \subsetneqq A}$ since the image of~${\varphi {\upharpoonright} [x, y]}$ contains points that correspond to inner points of edges. 
    This is a contradiction to the minimality of~$A$. 
\end{proof}

We conclude this section with the following corollary which allows us to extract a directed $s$--$t$~arc from a directed topological $s$--$t$~walk for distinct points~${s, t}$ of~$\abs{G}$. 

\begin{corollary}
    \label{cor:arc_in_walk}
    Let ${s, t \in V(G) \cup \Omega(G)}$ with~${s \neq t}$ for some locally finite weakly connected digraph~$G$. 
    Then every directed topological $s$--$t$~walk in $\abs{G}$ contains a directed $s$--$t$~arc. 
\end{corollary}

\begin{proof}
    Let~$\overline{W}$ be a directed topological $s$--$t$~walk with~${W \subseteq E(G)}$. 
    So~$W$ has property~\ref{item:walk2} of Lemma~\ref{lem:equi_lemma_walk}. 
    Now consider the set~$\mathcal{W}$ of all subsets of~$W$ that also have property~\ref{item:walk2} of Lemma~\ref{lem:equi_lemma_walk}. 
    This set is ordered by inclusion and not empty since~${W \in \mathcal{W}}$. 
    Next let us check that every decreasing chain~${C \subseteq \mathcal{W}}$ is bounded from below by~$\bigcap C$, which is an element of~$\mathcal{W}$. 
    Obviously, ${\bigcap C \subseteq c}$ holds for every~${c \in C}$. 
    To see that~${\bigcap C}$ is an element of~$\mathcal{W}$ note that for every finite cut~${E(X, Y)}$ of~$G$ with~${s \in \overline{X}}$ and~${t \in \overline{Y}}$ there exists a final segment~$C'$ of the decreasing chain~$C$ such that all~${c \in C'}$ contain the same edges from~${E(X, Y)}$. 
    As every~${c \in C}$ has also at least one edge from~${E(X, Y)}$, we know that the same is true for~${\bigcap C}$, which shows that~${\bigcap C \in \mathcal{W}}$ holds. 
    Now Zorn's Lemma implies that~$\mathcal{W}$ has a minimal element, which is a directed $s$--$t$~arc by Lemma~\ref{lem:equi_lemma_arc}. 
\end{proof}

\section{Packing pseudo arborescences}
\label{sec:pseudo-packing}

We begin this section with a lemma characterising when a packing of~${k \in \mathbb{N}}$ many edge-disjoint spanning $r$-reachable sets is possible in a locally finite weakly connected digraph~$G$ with~${r \in V(G) \cup \Omega(G)}$. 
This lemma is the main ingredient to prove our first main result. 
The proof is mainly based on a compactness argument. 

\begin{lemma}
    \label{lem:k_many-reachable}
    A locally finite weakly connected digraph~$G$ with~${r \in V(G) \cup \Omega(G)}$ has~${k \in \mathbb{N}}$ edge-disjoint spanning $r$-reachable sets if and only if every bipartition~${(X, Y)}$ of~${V(G)}$ with~${r \in \overline{X}}$ and~${\abs{E(X, Y)} < \infty}$ satisfies~${d^-(Y) \geq k}$. 
\end{lemma}

\begin{proof}
    The condition that every bipartition~${(X, Y)}$ of~${V(G)}$ with~${r \in \overline{X}}$ and~${\abs{E(X, Y)} < \infty}$ satisfies~${d^-(Y) \geq k}$ is obviously necessary for the existence of~$k$ edge-disjoint spanning $r$-reachable sets. 
    
    Let us now prove the converse. 
    First we fix a sequence~${(S_n)_{n \in \mathbb{N}}}$ of finite vertex sets~${S_n \subseteq V(G)}$ such that~${\bigcup_{n \in \mathbb{N}} S_n = V(G)}$. 
    For every~${n \in \mathbb{N}}$ let~$G_n$ denote the digraph which arises by contracting, inside of~$G$, each weak component of~${G-S_n}$ to a single vertex. 
    Here we keep multiple edges, but delete loops that arise. 
    Since~$G$ is locally finite, we know that each~$G_n$ is a finite digraph. 
    
    Note that, as in the proof of Lemma~\ref{lem:equi_lemma_walk}, if~${r \notin S_n}$, there exists a unique component~$C_n$ of~${G-S_n}$ such that~${r \in \overline{C_n}}$ and we refer to the vertex of~$G_n$ corresponding to~$C_n$ as~$r$. 
    
    Now we define~$V_n$ as the set of all $k$-tuples consisting of~$k$ edge-disjoint spanning \linebreak $r$-reachable sets of~$G_n$. 
    As every cut of~$G_n$ is finite and also corresponds to a cut of~$G$, our labelling with~$r$ ensures that each~$G_n$ has~$k$ edge-disjoint arborescences rooted in~$r$ by Theorem~\ref{thm:fin-edmonds}. 
    So none of the~$V_n$ is empty. 
    Furthermore, each~$V_n$ is finite as~$G_n$ is a finite digraph. 
    
    Next we show that every spanning $r$-reachable set~$F_{n+1}$ of~$G_{n+1}$ induces one for~$G_n$ via~${F_n := F_{n+1} \cap E(G_n)}$. 
    So let~$F_{n+1}$ be given and consider a cut~${E(X_n, Y_n)}$ of~$G_n$ with~${r \in X_n}$. 
    As each component of~${G-S_{n+1}}$ is contained in a component of~${G-S_{n}}$, we can find a cut~${E(X_{n+1}, Y_{n+1})}$ of~$G_{n+1}$ with~${r \in X_{n+1}}$ such that ${\overrightarrow{E}(X_n, Y_n) = \overrightarrow{E}(X_{n+1}, Y_{n+1})}$ (and in fact also ${\overrightarrow{E}(Y_n, X_n) = \overrightarrow{E}(Y_{n+1}, X_{n+1})}$). 
    Since~$F_{n+1}$ is a spanning $r$-reachable set of~$G_{n+1}$, we obtain that~$F_n$ is one of~$G_n$. 
    
    Now we can apply Lemma~\ref{lem:koenig} to the graph defined on the vertex set~${\bigcup_{n \in \mathbb{N}} V_n}$ where two vertices~${v_{n+1} \in V_{n+1}}$ and~${v_n \in V_n}$ are adjacent if the $i$-th spanning $r$-reachable set in~$v_n$ is induced by the $i$-th one of~${v_{n+1}}$ for every~$i$ with~${1 \leq i \leq k}$. 
    So we obtain a ray~${r_0r_1 \ldots}$ with~${r_n \in V_n}$ and set~${\mathcal{F} := (F^1, \ldots, F^k)}$ where~${F^i := \bigcup_{n \in \mathbb{N}} r^i_n}$ and~$r^i_n$ denotes the $i$-th entry of the $k$-tuple~$r_n$ for every~$i$ with~${1 \leq i \leq k}$. 
    Let us now check that each~$F^i$ is a spanning $r$-reachable set of~$G$. 
    As~${\bigcup_{n \in \mathbb{N}} S_n = V(G)}$ holds, we can find for every finite cut~${E(X, Y)}$ of~$G$ with~${r \in \overline{X}}$ an~${n \in \mathbb{N}}$ such that all endvertices of edges of~${E(X, Y)}$ are contained in~$S_n$. 
    Hence, there exists a cut~${E(X_n, Y_n)}$ of~$G_n$ with~${r \in X_n}$ such that ${\overrightarrow{E}(X_n, Y_n) = \overrightarrow{E}(X, Y)}$ and ${\overrightarrow{E}(Y_n, X_n) = \overrightarrow{E}(Y, X)}$. 
    Since each~$F^i$ contains the edges of~$r^i_n$, which is a spanning $r$-reachable set of~$G_n$ and, therefore, contains an edge of~${\overrightarrow{E}(X_n, Y_n)}$, we know that each~$F^i$ is a spanning $r$-reachable set of~$G$. 
    Finally, we get that all the~$F^i$ are pairwise edge-disjoint since for every~${n \in \mathbb{N}}$ the~$r^i_n$ are pairwise edge-disjoint. 
\end{proof}

The next lemma ensures the existence of pseudo-arborescences for a set ${Z \subseteq V(G) \setminus \lbrace r \rbrace}$ in the sense that every $r$-reachable set for~$Z$ contains one. 
The proof of this lemma works by an application of Zorn's Lemma and is very similar to the proof of Corollary~\ref{cor:arc_in_walk}. 
Therefore, we omit stating its proof. 

\begin{lemma}
    \label{lem:min_reach_exist}
    Let~$G$ be a locally finite weakly connected digraph and let~${Z \subseteq V(G) \setminus \lbrace r \rbrace}$ with~${r \in V(G) \cup \Omega(G)}$.
    Then each $r$-reachable set for~$Z$ in~$G$ contains a pseudo-arborescence for~$Z$ rooted in~$r$. 
    \qed
\end{lemma}

Combining Lemma~\ref{lem:k_many-reachable} and Lemma~\ref{lem:min_reach_exist} with~${Z = V(G) \setminus \lbrace r \rbrace}$ we now obtain one of our main results, Theorem~\ref{thm:weak-pseudo-edmonds}, 
which we now state in a slightly stronger version. 

\begin{theorem}
    \label{thm:pseudo-edmonds}
    A locally finite weakly connected digraph~$G$ with~${r \in V(G) \cup \Omega(G)}$ has~${k \in \mathbb{N}}$ edge-disjoint spanning pseudo-arborescences rooted in~$r$ if and only if every bipartition~${(X, Y)}$ of~${V(G)}$ with~${r \in \overline{X}}$ and~${\abs{E(X, Y)} < \infty}$ satisfies~${d^-(Y) \geq k}$. 
    \qed
\end{theorem}

\section{Structure of pseudo-arborescences}
\label{sec:pseudo-structure}

The following lemma characterises $r$-reachable sets in terms of directed arcs. 
Additionally, it justifies the naming of $r$-reachable sets. 

\begin{lemma}
    \label{lem:dir_rooted_arcs}
    Let~$G$ be a locally finite weakly connected digraph with sets~${F \subseteq E(G)}$ and~${Z \subseteq V(G) \setminus \lbrace r \rbrace}$ and let~${r \in V(G) \cup \Omega(G)}$. 
    Then~$F$ is an $r$-reachable set for~$Z$ in~$G$ if and only if there exists a directed $r$--$z$~arc inside~$\overline{F}$ for every~${z \in \overline{Z}}$.
\end{lemma}

\begin{proof}
    Let us first assume that~$F$ is an $r$-reachable set for~$Z$ in~$G$. 
    We fix some~${z \in \overline{Z}}$ and prove next that~${\abs{F \cap \overrightarrow{E}(X, Y)} \geq 1}$ holds for each finite cut~${E(X, Y)}$ with~${r \in \overline{X}}$ and~${z \in \overline{Y}}$. 
    If~$z$ is a vertex, this follows immediately from the definition of an $r$-reachable set for~$Z$. 
    In the case that~${z \in \Omega(G)}$, we also get that some vertex of~$Z$ lies in~$Y$. 
    This follows, because~$z$ is contained in the closed and, therefore, compact set~$\overline{Z}$, which implies the existence of a sequence~$S$ of vertices in~$Z$ converging to~$z$. 
    Since~${E(X, Y)}$ is a finite cut and~${z \in \overline{Y}}$, the set~${O := \abs{G} \setminus \overline{X \cup E(X,Y)}}$ is open, contained in~$\overline{G[Y]}$ and contains~$z$ by Lemma~\ref{lem:jumping-arc}. 
    Now~$O$ must contain infinitely many vertices of~$S$ and hence~$Y$ must do so as well. 
    Therefore, the desired inequality follows again by the definition of an $r$-reachable set for~$Z$. 
    
    Now we are able to use Lemma~\ref{lem:equi_lemma_walk}, which yields that~$\overline{F}$ contains a directed topological $r$--$z$~walk. 
    We complete the argument by applying Corollary~\ref{cor:arc_in_walk} telling us that~$\overline{F}$ contains also a directed $r$--$z$~arc. 
    
    Conversely, consider any finite cut~${E(X, Y)}$ with~${r \in \overline{X}}$ and~${Y \cap Z \neq \emptyset}$, say~${z \in Y \cap Z}$. 
    The assumption ensures the existence of a directed $r$--$z$~arc in~$\overline{F}$. 
    By Lemma~\ref{lem:equi_lemma_arc} we obtain that~${\abs{F \cap \overrightarrow{E}(X, Y)} \geq 1}$ holds as desired. 
\end{proof}

Now let us turn our attention towards spanning pseudo-arborescences rooted in some vertex or end in a locally finite weakly connected digraph. 
The question arises how similarly these objects behave compared to spanning arborescences rooted in some vertex in a finite graph. 
A basic property of finite arborescences is the existence of a unique directed path in the arborescence from the root to any other vertex of the graph. 
Closely related is the absence of any cycle, directed ones or even \textit{weak} ones, i.e.~cycles in the underlying undirected graph, in a finite arborescence since its underlying graph is a tree. 
Although we know by Lemma~\ref{lem:dir_rooted_arcs} that the closure of a spanning pseudo-arborescences contains a directed arc from the root to any other vertex (or even end) of the graph, we shall see in the following example that we can neither guarantee the uniqueness of such arcs nor avoid infinite circles (directed ones or \textit{weak} ones, i.e. circles of the underlying undirected graph).

\begin{example}
    \label{ex:multiple_arcs}
    Consider the graph depicted in Figure~\ref{fig:multiple_arcs}. 
    This graph contains spanning $r$-reachable sets, for example the bold black edges together with the bold grey edges. 
    However, every spanning $r$-reachable set of this graph must contain all bold black edges because for any head of such an edge there is no other edge of which it is a head. 
    As this graph has only one end, namely~$\omega$, we see that there are infinite circles, directed and weak ones, containing only bold black edges. 
    This shows already that, in general, it is not possible to find spanning $r$-reachable sets that do not contain directed or weak infinite circles. 
    So there does not exist a stronger version of Theorem~\ref{thm:pseudo-edmonds} in the sense that the edges of the underlying multigraph of every spanning pseudo-arborescences form a topological spanning tree in the Freudenthal compactification of the underlying multigraph.
    
    \begin{figure}[htbp]
        \centering
        \includegraphics[width=7.3cm]{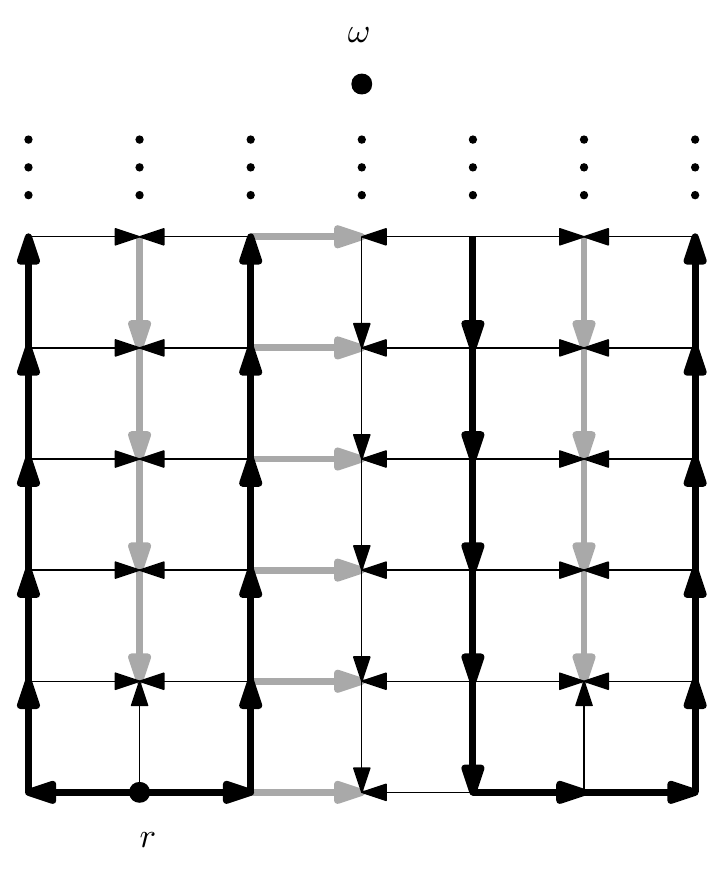}
        \caption{An example of a graph with a marked vertex~$r$ where the closure of any spanning $r$-reachable set contains an infinite circle and multiple arcs to the end~$\omega$ and certain vertices.}
        \label{fig:multiple_arcs}
    \end{figure}

    The graph in Figure~\ref{fig:multiple_arcs} shows furthermore that, in general, we cannot find spanning $r$-reachable sets~$F$ such that there exists a unique directed arc from $r$ to every vertex and every end of the graph inside~$\overline{F}$. 
    In the example we have two different directed arcs from~$r$ to the end~$\omega$ that contain only bold black edges and are therefore in every spanning $r$-reachable set of this graph. 
    Hence, we also get two different directed arcs from~$r$ to every vertex on the infinite directed circle that consists only of bold black edges. 
\end{example}

Although, in general, spanning pseudo-arborescences do not behave like trees in the sense that their underlying graphs correspond to topological spanning trees, they do so in a local sense. 
We conclude this section with our second main result, Theorem~\ref{thm:weak-r-reach_charaterisation}, characterising those spanning $r$-reachable sets that are inclusion-wise minimal via some local tree-like properties. 
In particular, we obtain the absence of finite cycles (directed or weak ones) in any spanning pseudo-arborescence.  
As mentioned before, we will prove a slightly stronger version of the theorem. 

\begin{theorem}
    \label{thm:r-reach_charaterisation}
    Let $G$ be a locally finite weakly connected digraph and~${r \in V(G) \cup \Omega(G)}$. 
    Then the following statements are equivalent for a spanning $r$-reachable set $F$ of $G$: 
    \begin{enumerate}
        [label=(\roman*)]
        \item \label{item:char1}
            $F$ is a spanning pseudo-arborescence rooted in $r$. 
        \item \label{item:char2}
            For every vertex $v \neq r$ of $G$ there is a unique edge in $F$ whose head is $v$, 
            and no edge in $F$ has $r$ as its head. 
        \item \label{item:char3}
            For every weak component $T$ of $G[F]$ the following holds: 
            If $r \in V(T)$, then $T$ is an arborescence rooted in $r$. 
            Otherwise, $T$ is an arborescence rooted in some end of $T$. 
    \end{enumerate}
\end{theorem}

\begin{proof}
    We start by proving the implication from~\ref{item:char1} to~\ref{item:char2}. 
    Let us first suppose for a contradiction that~$F$ contains an edge~$e$ whose head is~$r$. 
    Obviously, there is no finite cut~${E(X, Y)}$ of~$G$ such that~${r \in \overline{X}}$ and~${e \in \overrightarrow{E}(X, Y)}$. 
    Hence, ${F \setminus \lbrace e \rbrace}$ is a smaller spanning $r$-reachable set of~$G$ contradicting the minimality of~$F$. 
    
    Next let us consider an arbitrary vertex~${v \neq r}$ of~$G$. 
    We know by Remark~\ref{rem:basic_rem} that~$F$ contains at least one edge of ${\overrightarrow{E}(V(G) \setminus \lbrace v \rbrace, \lbrace v \rbrace)}$. 
    So~$F$ contains at least one edge whose head is~$v$. 
    
    Now suppose for a contradiction that there exists some vertex~${v \neq r}$ of~$G$ which is the head of at least two edges of~$F$, say~$e$ and~$f$. 
    We know by Lemma~\ref{lem:dir_rooted_arcs} that~$\overline{F}$ contains a directed $r$--$v$~arc~$A$. 
    Since the cut~${E(V(G) \setminus \lbrace v \rbrace, \lbrace v \rbrace)}$ is finite and~$A$ is a directed $r$--$v$~arc, we get that~$A$ must contain precisely one edge of~${\overrightarrow{E}(V(G) \setminus \lbrace v \rbrace, \lbrace v \rbrace)}$. 
    Hence, one of the edges~${e, f}$ is not contained in~$A$, say~$e$. 
    By the minimality of~$F$, we obtain that~${F \setminus \lbrace e \rbrace}$ cannot be a spanning $r$-reachable set of~$G$. 
    So there must exist a finite cut~${E(X, Y)}$ of~$G$ with~${r \in \overline{X}}$ such that~$e$ is the only edge in~${F \cap \overrightarrow{E}(X, Y)}$. 
    Now we have a contradiction since the head of~$e$ is~$v$ and lies in~$Y$, which means that the directed arc~$A$ contains at least one edge of~${\overrightarrow{E}(X, Y)}$ by Lemma~\ref{lem:equi_lemma_arc}, but such an edge is different from~$e$. 
    Therefore,~$e$ was not the only edge in~${F \cap \overrightarrow{E}(X, Y)}$. 
    
    We continue with the proof that statement~\ref{item:char2} implies statement~\ref{item:char3}.
    For this let us fix an arbitrary weak component~$T$ of~$G[F]$. 
    We now show that~$T$ is a tree. 
    Suppose for a contradiction that~$T$ contains a directed or weak cycle~$C$. 
    
    If~$C$ is a directed cycle, each vertex on~$C$ would already be a head of some edge of the cycle. 
    Hence,~$r$ cannot be a vertex on~$C$. 
    Applying Remark~\ref{rem:basic_rem} with the finite set~${V(C)}$, we obtain that there needs to be an edge~${uv}$ of~$F$ with~${v \in V(C)}$ and~${u \in V(G) \setminus V(C)}$. 
    So~$v$ is the head of two edges of~$F$, which contradicts statement~\ref{item:char2}. 
    
    In the case that~$C$ is a cycle, but not a directed one, take a maximal directed path on~$C$.
    Its endvertex is the head of two edges of~$C$. 
    So we get again a contradiction to statement~\ref{item:char2}. 
    We can conclude that~$T$ is a tree.  
    
    If~$r$ is a vertex of~$T$, then it is immediate from statement~\ref{item:char2} that~$T$ is an arborescence rooted in~$r$. 
    Otherwise, there needs to be a backwards directed ray~$R$ in~$T$ as each vertex different from~$r$ is the head of a unique edge of~$F$. 
    Let~$\omega$ be the end of~$T$ which contains~$R$. 
    Hence,~$T$ is an arborescence rooted in~$\omega$, completing the proof of this implication. 
    
    It remains to show the implication from~\ref{item:char3} to~\ref{item:char1}. 
    For this we assume statement~\ref{item:char3} and suppose for a contradiction that~$F$ is not minimal with respect to inclusion. 
    Hence, ${F' = F \setminus \lbrace e \rbrace}$ is a spanning $r$-reachable set as well for some~${e = uv \in F}$. 
    Let~$T$ be the weak component of~${G[F]}$ which contains~ $v$. 
    As~$T$ is an arborescence rooted in~$r$ or some end of~$T$, we get that no edge of~$F'$ has~$v$ as its head. 
    Note that~${r \neq v}$ because of the edge~${uv \in F}$. 
    Now we get a contradiction by applying Remark~\ref{rem:basic_rem} with~$F'$ and the set~${\lbrace v \rbrace}$, which tells us that~$F'$ needs contains an edge whose head is~$v$. 
\end{proof}

The question might arise whether we can be more specific in statement~\ref{item:char3} of Theorem~\ref{thm:r-reach_charaterisation} in the case when~$r$ is an end of~$G$. 
Unfortunately, it is not true that there has to exist a weak component of~${G[F]}$ whose unique backwards directed ray lies in~$r$. 
The reason for this is that the end~$r$ might be an accumulation point of a sequence of infinitely many different weak components of~${G[F]}$ in~$\abs{G}$ each of which contains a backwards directed ray to a different end of~$G$. 
It is not difficult to construct an example for this situation and so we omit such a description here. 
On the other hand if the end~${r \in \Omega(G)}$ is not an accumulation point of different ends of~$G$, then there exists at least one weak component of~${G[F]}$ whose backwards directed ray is contained in~$r$. 
To see this fix an arbitrary directed $r$--$v$~arc~$A$ inside~$\overline{F}$ for some vertex~$v$. 
Since~$F$ is a spanning $r$-reachable set of~$G$, we can find such an arc. 
If among all of the weak components of~${G[F]}$ which are met by~$A$, there is a first one with respect to the linear order of~$A$, then a backwards directed ray of this component is an initial segment of~$A$ and, therefore, contained in~$r$.
Note for the other case that tails of the backwards directed rays of each component of~${G[F]}$ that is met by~$A$ must be contained in~$A$. 
Since~$A$ is an arc, all these backwards directed rays must be contained in different ends of~$G$. 
These ends, however, would then have~$r$ as an accumulation point in~$\abs{G}$ contradicting the assumption on~$r$.

\section*{Acknowledgements}
J. Pascal Gollin was supported by the Institute for Basic Science (IBS-R029-C1).

Karl Heuer was partly supported by the European Research Council (ERC) under the European Union's Horizon 2020 research and innovation programme (ERC consolidator grant DISTRUCT, agreement No.\ 648527).

\end{document}